\documentclass[11pt]{amsart}
\usepackage{amssymb,latexsym,amsmath,amsthm}
\usepackage{graphicx}
\usepackage{hyperref}
\usepackage[justification=centering]{caption}
\usepackage{mathrsfs}
\usepackage{ulem}

\author{Rachel Barber}
\address{Rachel Barber, Department of Mathematics, Hood College\\  401 Rosemont Ave, Frederick, MD 21701, USA}
\email{rbarber@hood.edu}
\author{Ted Dobson}
\address{Ted Dobson, University of Primorska, UP IAM, Muzejski trg 2, SI-6000 Koper, Slovenia and\\ University of Primorska, UP FAMNIT, Glagolja\v{s}ka 8, SI-6000 Koper, Slovenia}
\email{ted.dobson@upr.si}
\thanks{This work is supported in part by the Slovenian Research Agency (research program P1-0285 and research projects N1-0140, N1-0160, J1-2451, N1-0208, J1-3001, J1-3003, J1-4008, and J1-50000).}

\usepackage[none]{hyphenat}

\usepackage{tikz}
\usetikzlibrary{graphs,graphs.standard,quotes}
\usetikzlibrary{arrows.meta}
\tikzstyle{V}=[draw, fill =black, circle, inner sep=0pt, minimum size=4pt]

\usepackage[margin=1.0in]{geometry}
\usepackage{fancyhdr}
\usepackage{lipsum}

\theoremstyle{plain}
\numberwithin{equation}{section}
\newtheorem{theorem}{Theorem}[section]
\newtheorem{corollary}[theorem]{Corollary}

\newtheorem{lemma}[theorem]{Lemma}

\theoremstyle{definition}
\newtheorem*{solution*}{Solution}
\newtheorem{definition}[theorem]{Definition}

\newtheorem{problem}[theorem]{Problem}

\DeclareMathOperator{\B}{\mathcal{B}}

\def\fix{{\rm fix}}
\def\Cay{{\rm Cay}}
\def\Aut{{\rm Aut}}

\def\tl{\triangleleft}
\def\la{{\langle}}
\def\ra{{\rangle}}
\def\Z{{\mathbb Z}}
\def\cal{\mathcal}
\def\Stab{{\rm Stab}}
\def\Reducible{{\rm Reducible}}
\def\core{{\rm core}}
\def\Cos{{\rm Cos}}

\begin{document}
\title{Finding automorphism groups of double coset graphs and Cayley graphs are equivalent}

\begin{abstract}
It has long been known that a vertex-transitive graph $\Gamma$ is isomorphic to a double coset graph $\Cos(G,H,S)$ of a transitive group $G\le\Aut(\Gamma)$, a vertex stabilizer $H\le G$, and some subset $S\subseteq G$.  We show that the automorphism group of the Cayley graph $\Cay(G,S)$ with connection set $S$ can be obtained from the automorphism group of $\Cos(G,H,S)$ and vice versa.  We also show that the isomorphism problem for double coset graphs is equivalent to the isomorphism problem for Cayley graphs provided one knows all groups $G$ for which a fixed Cayley graph is a Cayley graph of $G$.  Our main tool is a ``recognition theorem", which recognizes when a Cayley graph of a group $G$ is a wreath product of two graphs based upon its connection set.
\end{abstract}

\maketitle

\section{Introduction}

The problem of determining the automorphism group of a vertex-transitive graph is perhaps the most important and central problem in algebraic graph theory.  In this paper, we will show in Corollary \ref{main cor} that the problem of determining the automorphism group of a double coset graph $\Cos(G,H,S)$ is equivalent to the problem of determining the automorphism group of its corresponding Cayley graph, $\Cay(G,S)$.  As every vertex-transitive graph is isomorphic to a double coset graph \cite[Theorem 2]{Sabidussi1964}, the seemingly larger problem of determining automorphism groups of vertex-transitive graphs is equivalent to the problem of finding the automorphism groups of Cayley graphs (assuming one knows a transitive subgroup of the graphs' automorphism group).  We also show a similar result concerning isomorphisms between vertex-transitive graphs and Cayley graphs, but we must also know all groups $G$ for which the Cayley graph is Cayley graph of $G$.  All results in this paper are true not just for graphs, but also digraphs.

The fundamental tool that we will use is a ``recognition theorem" for double coset graphs that can be written as a wreath product.  That is, in Theorem \ref{cosetwr} we find necessary and sufficient conditions on the connection set $S$ of a double coset graph of $G$ to be isomorphic to a wreath product of two smaller graphs. It is known that a Cayley digraph $\Cay(A,S)$ of an abelian group $A$ is isomorphic to a wreath product of two smaller digraphs if and only if there exists $1 < B < A$ such that $S \setminus B$ is a union of cosets of $B$.  This was shown explicitly for prime powers in \cite{KovacsS2012, Morris1999} and mentioned without proof in \cite{BhoumikDM2014}.  Theorem \ref{cosetwr} can be used to show a bijective correspondence between the sets of double coset graphs of $G$ with a natural property and reducible Cayley graphs of $G$ (Corollary \ref{bijection}).  This correspondence gives a more ``graph theoretic" and intuitive definition (Definition \ref{coset alt definition}) of a double coset graph of $G$ as a quotient of a Cayley graph of $G$, and shows that double coset graphs can be interpreted as nothing more than devices for succinctly storing the symmetry information of some Cayley graphs of $G$.

To finish, we are able to extend the definition of generalized wreath product digraphs to all double coset digraphs.  Generalized wreath product digraphs are a relatively new but very important family of digraphs from the point of view of computing automorphisms groups, and previously were only defined for Cayley digraphs of abelian groups precisely because the recognition problem of wreath products was only solved for Cayley digraphs of abelian groups.  This last problem was the original motivation for this work.

\section{Digraphs as wreath products}

In this section, we collect the basic definitions and results we will need.  For most, additional information as well as illustrative examples can be found in \cite{Book}.

\begin{definition}
Let $\Gamma_1$ and $\Gamma_2$ be digraphs. The \textbf{wreath product of $\Gamma_1$ and $\Gamma_2$}, denoted $\Gamma_1\wr\Gamma_2$,  is the digraph with vertex set $V(\Gamma_1)\times V(\Gamma_2)$ and arc set
$$\{ ((u,v)(u,v')):u\in V(\Gamma_1){\rm\ and\ } (v,v')\in A(\Gamma_2)\}\cup\{((u,v)(u',v')): (u,u')\in A(\Gamma_1){\rm\ and\ }v,v'\in V(\Gamma_2)\}.$$
\end{definition}

The wreath product is sometimes referred to  as the lexicographic product, graph composition, or the $\Gamma_2$-extension of $\Gamma_1$.

\begin{definition}
Let $X$ be a set, $G \leq S_X$ be a transitive group, and $B \subseteq X$. We call $B$ a \textbf{block} of $G$ if whenever $g \in G$, then $g(B) \cap B = \emptyset$ or $B$. If $B = \{x\}$ for some $x \in X$ or $B = X$, then $B$ is a \textbf{trivial block.}
\end{definition}

Note that if $B$ is a block of $G$, then so is $g(B)$ for every $g \in G$, and is called a \textbf{conjugate block of $B$}. The set of all blocks conjugate to $B$, denoted $\mathcal{B}$, is a partition of $X$, and is called a \textbf{block system of $G$}.  If $\B$ is the set orbits of a normal subgroup of $G$, it is called a \textbf{normal} block system of $G$.

\begin{definition}
Let $G\le S_X$ and $H\le S_Y$.  Define the {\bf wreath product of $G$ and $H$}, denoted $G\wr H$, to be the set of all permutations of $X\times Y$ of the form $(x,y)\mapsto (g(x),h_x(y))$, where $g\in G$ and each $h_x\in H$.
\end{definition}

It is straightforward to show that $G\wr H$ is a group.  We will have need of a special block system of the wreath product of two transitive permutation groups.

\begin{definition}
Let $G \leq S_X$ and $H \leq S_Y$ be transitive groups. The \textbf{lexi-partition} of $G\wr H$ with respect to $H$ is the block system $\B = \{\{(x,y) : y \in Y\} : x \in X\}$.
\end{definition}

A transitive permutation group with a block system has an induced action on the block system.

\begin{definition}
Suppose that $G \leq S_n$ is a transitive group with a block system $\B$. Then $G$ has an \textbf{induced action on $\B$}, which we denote by $G/\B$. Namely, for $g \in G$, we define $g/\B(B) = B^\prime$ if and only if $g(B) = B^\prime$, and set
$G/\B = {\{g/\B : g \in G\}}$. We also define the \textbf{fixer of $\B$ in $G$}, denoted $\text{fix}_G(\B)$, to be $\{g \in G : g/\B = 1\}$.
\end{definition}

So $\fix_G({\cal B})$ is the kernel of the induced action of $G$ on ${\cal B}$.

\begin{definition}\label{block quotient defin}
\begin{sloppypar}
Let $\Gamma$ be a vertex-transitive digraph whose automorphism group contains a transitive subgroup $G$ with a block system $\B$. Define the \textbf{block quotient digraph of $\Gamma$ with respect to $\B$}, denoted $\Gamma/\B$, to be the digraph with vertex set $\B$ and arc set $A(\Gamma/\B) = {\{(B, B^\prime) : B \not= B^\prime \in \B \text{ and } (u,v) \in A(\Gamma), u \in B, v \in B^\prime\}.}$
\end{sloppypar}
\end{definition}

So in a block quotient digraph, a block $B$ is out-adjacent to a block $B'$ if and only if some vertex of $B$ is out-adjacent to some vertex of $B'$.

The following result gives necessary and sufficient conditions to recognize that a vertex-transitive digraph is a wreath product (but not from its connection set). The first version of this result is by Joseph \cite[Lemma 3.11]{Joseph1995}. The following is a generalization of her result, whose very similar proof can be found in \cite[Theorem 4.2.15]{Book}.

\begin{lemma}\label{wrthm}
Let $\Gamma$ be a vertex-transitive digraph whose automorphism group contains a transitive subgroup $G$ that has a block system $\B$. Then $\Gamma \cong \Gamma/\B \wr \Gamma[B_0]$, $B_0 \in \B$, if and only if whenever $B, B^\prime \in \B$ are distinct then there is an arc $(x, x^\prime)$ from a vertex $x \in B$ to a vertex $x^\prime \in B^\prime$ if and only if every arc of the form $(x, x^\prime)$ with $x \in B$ and $x^\prime \in B^\prime$ is contained in $A(\Gamma).$
\end{lemma}

The following theorem was proven in \cite[Theorem 5.7]{DobsonM2009}, and gives the automorphism group of the wreath product of finite vertex-transitive digraphs.

\begin{theorem}\label{autwr}
Let $\Gamma_1$ and $\Gamma_2$ be finite vertex-transitive digraphs. If $\Gamma = \Gamma_1 \wr \Gamma_2$ and $\Aut(\Gamma) \not= \Aut(\Gamma_1) \wr \Aut(\Gamma_2)$, then there exist positive integers $r >1$ and $s >1$ and vertex-transitive digraphs $\Gamma_1^\prime$ and $\Gamma_2^\prime$ for which $\Gamma_1 \cong \Gamma_1^\prime \wr K_r, \; \Gamma_2 \cong K_s \wr \Gamma_2^\prime$ or $\Gamma_1 \cong \Gamma_1^\prime \wr \bar{K_r}, \; \Gamma_2 \cong \bar{K_s} \wr \Gamma_2^\prime$. In either case, $\Aut(\Gamma) \cong \Aut(\Gamma_1^\prime) \wr S_{rs} \wr \Aut(\Gamma_2^\prime).$
\end{theorem}

Suppose $\Aut(\Gamma_1\wr\Gamma_2)\not = \Aut(\Gamma_1)\wr\Aut(\Gamma_2)$ and $\Gamma_1\wr\Gamma_2$ is neither complete nor the complement of a complete graph; that is, assume that in the statement of Theorem \ref{autwr}, we have that $\Gamma_1^\prime$ or $\Gamma_2^\prime$ has more than one vertex.  Theorem \ref{autwr} implies that $\Gamma_1\wr\Gamma_2$ can be written as another wreath product of two vertex transitive digraphs whose automorphism group is the wreath product of the automorphism groups. That is, if $\Gamma \cong \Gamma_1 \wr \Gamma_2$ with $\Aut(\Gamma) \not\cong \Aut(\Gamma_1) \wr \Aut(\Gamma_2)$, then there are two nontrivial digraphs $\Gamma_3 , \Gamma_4$ where $\Gamma \cong \Gamma_3 \wr\Gamma_4$ with $\Aut(\Gamma) = \Aut(\Gamma_3)\wr\Aut(\Gamma_4)$.  For example, if $\Gamma_1 \cong \Gamma_1^\prime \wr K_r$ and $\Gamma_2 \cong K_s \wr \Gamma_2^\prime$, then we may take $\Gamma_3 = \Gamma_1^\prime$ and $\Gamma_4 = K_{rs}\wr\Gamma_2^\prime$, while if $\Gamma_1 \cong \Gamma_1^\prime \wr \bar{K_r}$ and $\Gamma_2 \cong \bar{K_s} \wr \Gamma_2^\prime$ then we may take $\Gamma_3 = \Gamma_1^\prime$ and $\Gamma_4 = \bar{K}_{rs}\wr\Gamma_2^\prime$.



We now give the definitions of double coset digraphs and Cayley digraphs, after giving the definition of double cosets.

\begin{definition} Let $G$ be a group and $H, K < G$. For each $g \in G$, the set $HgK = \{hgk : h \in H, k \in K\}$ is called the \textbf{$(H, K)$-double coset of $g$ in $G$}.  If $H = K$, then the $(H, K)$-double cosets are referred to as the \textbf{double cosets of $H$ in $G$}.
\end{definition}

The $(H, K)$-double cosets of $G$ form a partition of $G$ and need not have the same cardinality. Note that the $(\{1\}, K)$-double cosets are the left cosets of $K$ in $G$ and the $(H, \{1\})$-double cosets are just the right cosets of $H$ in $G$. In general, $HgK$ is a union of right cosets of $H$ as well as a union of left cosets of $K$.

Throughout the paper, we denote the set of left cosets of $H$ in $G$ by $G/H$.

\begin{definition}
Let $G$ be a group, $H \leq G$, and $S \subset G$ such that $S\cap H = \emptyset$ and $HSH = S$. Define a digraph $\text{Cos}(G,H,S)$ with vertex set $V(\text{Cos}(G,H,S)) = G/H$ the set of left cosets of $H$ in $G$, and arc set $A(\text{Cos}(G,H,S)) = \{(gH,gsH):g\in G{\rm\ and\ }s\in S\}$. The digraph $\text{Cos}(G,H,S)$ is called the \textbf{double coset digraph of $G$} with \textbf{connection set $S$} (or $HSH$).
\end{definition}

It is customary to impose on $H$ the condition that it is core-free in $G$.  That is, that it contains no nontrivial normal subgroups of $G$.  This ensures that the action of $G$ on the left cosets of $H$ in $G$ is faithful, which is certainly a condition one would want if one were, say, drawing a particular double coset digraph.  We will NOT follow this custom - several of the applications of our main results are actually false if this convention is followed.  The custom also, inconveniently for us, means that abelian groups $A$, for example, have NO double coset digraphs of $A$ that are not Cayley digraphs of $A$ as the only core-free subgroup of an abelian group is the trivial group.

It is also customary to insist that $S\cap H = \emptyset$, as this means that double coset digraphs do not have loops.  In some sense, it does not matter at all whether or not one allows loops.  In fact, the results in the next section are completely independent of this choice, although the condition that $HSH = S$ forces $S\cap H = \emptyset$ or $S\cap H = H$, and so each vertex would have $\vert H\vert$ loops at a vertex (to obtain $r < \vert H\vert$ loops one would insist that $H(S\setminus H)H = S\setminus H$, and then insist that $\vert H\cap S\vert = r$).  The results in Section 4 though, have statements that do depend upon $S\cap H = \emptyset$.  By this, we mean that if one allowed loops, those results are not true ``as is", but it would be a simple matter to modify them appropriately.  We follow this custom for two reasons.  One is the custom (and indeed this condition goes back to Sabidussi's original work on double coset graphs).  The other is that the results in Section 4 are easier to state in this form.

\begin{definition}
Let $G$ be a group and $S \subset G$ such that $1\not\in S$. Define a \textbf{Cayley digraph of $G$}, denoted $\text{Cay}(G,S)$, to be the digraph with vertex set $V(\text{Cay}(G,S)) = G$ and arc set $A(\text{Cay}(G,S)) = \{(g,gs) : g \in  \nobreak G,  s \in S\}.$ We call $S$ the \textbf{connection set of $\text{Cay}(G,S)$}.
\end{definition}

So a Cayley digraph of $G$ is simply the special case of a double coset digraph of $G$ with $H = 1$.

\section{Double coset digraphs as wreath products}

We begin with a result which is probably known by many readers, but which we will state and prove anyway due to its importance.

\begin{lemma}\label{block form}
Let $G$ be a group, $H\le G$, and let $G$ act on $G/H$ by left multiplication.  Then $G$ acts transitively on $G/H$, and any block system ${\cal B}$ of this action has blocks the set of left cosets of $H$ contained in some left coset of a subgroup $H \le K\le G$.
\end{lemma}

\begin{proof}
That $G$ is transitive on $G/H$ is well known.  By \cite[Theorem 1.5A]{DixonM1996}, there is $K\ge H$ such that the block $B\in{\cal B}$ that contains $H$ is the orbit of $K$ that contains $H$.  Then $B$ is the set of all left cosets of $H$ contained in $K$.  Any block conjugate to $B$ is simply of the form $gB$ for some $g\in G$, and so is the set of left cosets of $H$ contained in $gK$.
\end{proof}

\begin{lemma}\label{K determined lemma}
Let $G$ be a group, $H\le G$, and $S\subset G$ such that $S\cap H = \emptyset$ and $S = HSH$.  Suppose $\Gamma = \Cos(G,H,S) \cong \Gamma_1\wr\Gamma_2$ for $\Gamma_1$ and $\Gamma_2$ vertex-transitive digraphs  with at least two vertices.  Let ${\cal B}$ be the lexi-partition of $\Aut(\Gamma_1)\wr\Aut(\Gamma_2)$ with respect to $\Aut(\Gamma_2)$. If ${\cal B}$ is the set of left cosets of $K$ in $G$, then $S\setminus K= K(S\setminus K)K$.
\end{lemma}

\begin{proof}
By Lemma \ref{block form}, ${\cal B}$ is the set of left cosets of some subgroup $H\le K\le G$.  Suppose $a\in S\setminus K$.  Let $B = aK\in{\mathcal B}$ (here $aK$ is viewed as a union of left cosets of $H$ in $G$).  Then $H \not\subseteq B$ as $H\le K$.  Since there is an arc from $H$ to $aH$  if and only if there is an arc from $H$ to every left coset of $H$ contained in $B$ as $\Gamma = \Gamma_1\wr\Gamma_2$ and ${\cal B}$ is the lexi-partition with respect to $\Aut(\Gamma_2)$, we conclude that $aK\subseteq S = HSH$, and so $H(S \setminus K)H$ is a union of left cosets of $K$ in $G$.  Then, for every $k , k^\prime \in K$, we have $kH, k^\prime H \subseteq K$, and $(kH, ak^\prime H) \in A(\text{Cos}(G,H,S))$. Hence $k^{-1}ak^\prime \in HSH$ for every $k,k'\in K$. So $KaK\subseteq HSH$.  As $a \not\in K$,  $k^{-1}ak^\prime \in H(S\setminus K)H$. Thus, $H(S \setminus K)H$ is a union of double cosets of $K$ in $G$.
\end{proof}

With the next result, we are able to recognize double coset digraphs that are isomorphic to nontrivial wreath products from their connection sets.

\begin{theorem}\label{cosetwr}
Let $G$ be a group, $H\le G$, and $S\subset G$ such that $S\cap H = \emptyset$ and $HSH = S$.  The double coset digraph $\Gamma = \Cos(G,H,S)$ is isomorphic to a nontrivial wreath product of two vertex-transitive digraphs of smaller order if and only if there exists $H < K < G$ such that $H(S\setminus K)H$ is a union of double cosets of $K$ in $G$. If such a $H < K < G$ exists and $\mathcal{B}$ is the set of left cosets of $K$ in $G$, then
\begin{eqnarray}\label{first equation}
\Cos(G,H,S) \cong \Gamma/\mathcal{B} \wr \Gamma[K] \cong \Cos(G,K,S\setminus K) \wr \Cos(K,H,(S\cap K)).
\end{eqnarray}
Additionally, if $\Gamma$ is neither complete nor the complement of a complete graph and $K$ is chosen to be maximal in $G$ with the above properties, then
\begin{eqnarray}\label{autgrpcoset}
\Aut(\Cos(G,H,S)) \cong \Aut(\Cos(G, K, S \setminus K))\wr \Aut(\Cos(K,H,(S\cap K))).
\end{eqnarray}
\end{theorem}

\begin{proof}
Suppose $\Gamma = \Gamma_1\wr\Gamma_2$, where both $\Gamma_1$ and $\Gamma_2$ are nontrivial and not $\Gamma$.  Then the lexi-partition ${\cal B}$ of $\Aut(\Gamma_1)\wr\Aut(\Gamma_2)$ with respect to $\Aut(\Gamma_2)$ is a block system of $\Aut(\Gamma_1)\wr\Aut(\Gamma_2)\le\Aut(\Gamma)$. The result follows by Lemma \ref{K determined lemma}.

Conversely suppose $H <K <G$ such that $H(S \setminus K)H$ is a union of double cosets of $K$ in $G$. We now show $\Gamma\cong\Gamma/{\cal B}\wr\Gamma[K]$. This will complete the ``if and only if" part of the proof as well as the first part of Equation \ref{first equation}.  Suppose $(ak_1H, bk_2H) \in A(\Gamma)$ and $ak_1H$ and $bk_2H$ are not contained in the same left coset of $K$ in $G$, where $k_1,k_2\in K$.  This gives $k_1^{-1}a^{-1}bk_2\not\in K$.  Then $k_1^{-1}(a^{-1}b)k_2 \in K(S \setminus K)K = H(S \setminus K)H$, with the last equality holding as $H(S\setminus K)H$ is a union of double cosets of $K$ in $G$. Thus $K(a^{-1}b)K \subseteq K(S \setminus K)K$ and $(ak'H, bkH) \in A(\Gamma)$ for all $k, k' \in K$.  This means that whenever $B, B^\prime \in \B$ are distinct then there is an arc $(x, x^\prime)$ from a vertex $x \in B$ to a vertex $x^\prime \in B^\prime$ if and only if every arc of the form $(x, x^\prime)$ with $x \in B$ and $x^\prime \in B^\prime$ is contained in $A(\Gamma).$ By Lemma \ref{wrthm}, $\Gamma \cong \Gamma/\mathcal{B} \wr \Gamma[K]$.

We next show $\Gamma/{\cal B} = \Cos(G,K,S\setminus K)$.  The digraph $\Cos(G,K,S\setminus K)$ is a well-defined double coset digraph as $S\setminus K$ is a union of double cosets of $K$.  Let $a,b\in G$ such that $a^{-1}b\not\in K$ (or $a^{-1}b\in S\setminus K$).  Then $(aK,bK)\in A(\Gamma/{\cal B})$ if and only if there is $a_1,b_1\in G$ such that $a_1H\subseteq aK$, $b_1H\subseteq bK$, and $(a_1H,b_1H)\in A(\Gamma)$.  This occurs if and only if $a_1^{-1}b_1H\subseteq S$.  As $S\setminus K$ is a union of double cosets of $K$ and $a_1^{-1}b_1^{-1}H\subseteq a^{-1}bK$, we see $a_1^{-1}b_1K\in S\setminus K$ (viewing $S\setminus K$ as a union of left cosets of $K$ in $G$).  Thus $(aK,bK)\in A(\Gamma/{\cal B})$ if and only if $a^{-1}bK\in S\setminus K$, which occurs if and only if $(aK,bK)\in A(\Cos(G,K,S\setminus K))$.  So $\Gamma/{\cal B} = \Cos(G,K,S\setminus K)$.

As $K$ is a left coset of itself, $K \in \B$. Then $\Cos(G,H,S)[K] = \Cos(K,H,K\cap S)$, and so $\Gamma[K] = \Cos(K,H,S\cap K)$.
This completes the proof of Equation \ref{first equation}.

It now only remains to show that if $\Gamma$ is neither complete nor the complement of a complete graph and $K$ is maximal with the property that $H(S\setminus K)H$ is a union of double cosets of $K$, then Equation \eqref{autgrpcoset} holds.
Suppose otherwise.  Then $\Cos(G,K,S\setminus K)$ can be written as a nontrivial wreath product by Theorem \ref{autwr} and the comment following it.  By what we have already shown, there is some $K < K' < G$ such that $S\setminus K'$ is a union of double cosets of $K'$ in $G$, contradicting the maximality of $K$.
\end{proof}

If one follows the convention that $H$ is core-free in $G$ for double coset digraphs, then with $N = \core_G(H)$, ${\cal C}$ the set of left cosets of $K/N$ in $G/N$, and $T = \{(sN)(K/N) : s \in H(S \setminus K)H\}$, Equation \ref{autgrpcoset} becomes
\begin{eqnarray*}
\Cos(G/N,H/N,\{sN:s\in S\}) & \cong & \Gamma/\mathcal{C} \wr \Gamma[K/N]\\
                            & \cong & \Cos(G/N,K/N,T) \wr \Cos(K/N,H/N,(S\cap K)N).
\end{eqnarray*}

When $H = \{1_G\}$, $\Cos(G,H,S) \cong \Cay(G,S)$ and we have a special case of Theorem \ref{cosetwr} for Cayley digraphs.

\begin{corollary}\label{nonabelwr}
Let $G$ be a group and $S\subset G$ such that $1\not\in S$.  The Cayley digraph $\Gamma = \Cay(G,S)$ is isomorphic to a nontrivial wreath product of two vertex-transitive digraphs of smaller order if and only if there exists $1 < K < G$ such that $S \setminus K$ is a union of double cosets of $K$ in $G$. If such a $1 < K < G$  exists and $\mathcal{B}$ is the set of left cosets of $K$ in $G$, then
$$\Cay(G,S) \cong \Gamma/\mathcal{B} \wr \Gamma[K] \cong \Cos(G,K,S\setminus K) \wr \Cay(K,S \cap K).$$
Additionally, if $\Gamma$ is neither complete nor the complement of a complete graph and $K$ is chosen to be maximal in $G$ with the above properties, then
$$\Aut(\Cay(G,S)) \cong \Aut(\Cos(G,K,S\setminus K))\wr \Aut(\Cay(K,S \cap K)).$$
\end{corollary}

We note that in \cite{Chudnovsky2024} it was shown that for the special case of Cayley graphs, we only need that $SH = S$. They also mention that for graphs the condition that $SH = S$ and $HSH = S$ are equivalent. This is easy to see as if $SH = S$ and $S^{-1} = S$, then $SH = S$ implies that $(SH)^{-1} = H^{-1}S^{-1} = HS = H(SH) = HSH$.
If, in the previous result, we have $K\tl G$, then we get a slightly nicer sufficient (but not necessary) condition for a Cayley graph of $G$ to be a wreath product.

\begin{corollary}\label{wrnorm}
Let $G$ be a group and $S\subset G$ such that $1\not\in S$.  The Cayley digraph $\Gamma = \Cay(G,S)$ is isomorphic to a nontrivial wreath product of two vertex-transitive digraphs of smaller order if there exists $1 < K \triangleleft G$ such that $S \setminus K$ is a union of cosets of $K$ in $G$. In this case, if ${\cal B}$ is the set of cosets of $K$ in $G$, then
$$\Cay(G,S) \cong \Gamma/\B \wr \Gamma[K] \cong \Cay(G/K,S_1) \wr \Cay(K,S_2)$$
where $S_1$ is the set of cosets of $K$ contained in $S$ and $S_2 = K \cap S.$  Additionally, if $\Gamma$ is not complete nor the complement of a complete graph and $K$ is chosen to be maximal in $G$ with the above properties, then
$$\Aut(\Cay(G,S)) \cong \Aut(\Cay(G/K,S_1))\wr \Aut(\Cay(K,S_2).$$
\end{corollary}

\begin{proof}
As $K\tl G$, $G/K$ is a group of order $[G:K]$.  Also, $\Stab_K(1) = 1$, and so $K$ is semiregular of order, say, $k$.  Let $n = \vert G\vert$.  Then $[G:K] = n/k$ and $\fix_G({\cal B}) = \{k_L:k\in K\}$ has order $k$.  As $G_L/{\cal B}$ is transitive of order $n/k$, we see $G_L/{\cal B} = (G/K)_L$ is regular.  Hence $\Gamma/{\cal B}$ is a Cayley digraph of $G/K$.  As $G\setminus K$ is a union of left cosets of $K$ in $G$ (as it is a union of double cosets of $K$), we see that $\Gamma/{\cal B}$ has connection set $S_1$.  The rest follows by Corollary \ref{nonabelwr}.
\end{proof}

In the case when {\it all} subgroups of $G$ are normal, the above sufficient condition is also necessary.  Groups in which every subgroup are normal are called {\bf Dedekind groups}.  Obviously, abelian groups are Dedekind groups, and non-abelian Dedekind groups are called the {\bf Hamilton groups}.  Hamilton groups have the form $G = Q_8\times B\times D$ \cite[Theorem 12.5.4]{Hall1976}, where $Q_8$ is the quaternion group of order $8$, $B$ is an elementary abelian $2$-group, and $D$ is a finite abelian group of odd order.  This result also generalizes the known result mentioned earlier that Cayley digraphs of abelian groups can be written as nontrivial wreath products of two digraphs of smaller order if and only if its connection set is a union of cosets of some subgroup.

\section{Applications to double coset digraphs}

Sabidussi has shown that there is a strong relationship between Cayley graphs and double coset graphs.  He showed in \cite[Theorem 4]{Sabidussi1964} that $\Cos(G,H,S)\wr \bar{K}_n$, where $n = \vert H\vert$, is isomorphic to a Cayley graph $\Gamma$ of $G$ ($\Gamma$ has connection set a union of left cosets of $H$ in $G$).  He also showed in \cite[Theorem 2]{Sabidussi1964} that every double coset graph of $G$ is the quotient of a Cayley graph of $G$ with connection set $S$ by the partition ${\cal P}$ of $G$ given by the left cosets of a subgroup of $G$ that is disjoint from $S$.  Examining the proofs of these two results (but not the statements), it turns out the two Cayley graphs Sabidussi constructs in \cite[Theorems 2 and 4]{Sabidussi1964} are equal.

In the first part of this section, we more or less continue where Sabidussi left off, and tighten and make clearer the relationship between Cayley digraphs of $G$ and double coset digraphs of $G$, culminating in a new, graph theoretic definition of a double coset digraph of $G$ in Definition \ref{coset alt definition}.  For the rest of this section, we exploit Theorem \ref{cosetwr}, which intuitively allows for the use of a stronger quotient than the block quotient given in Definition \ref{block quotient defin}.  The block quotient, which Sabidussi also used, maps arcs of $\Gamma$ to arcs of $\Gamma/{\cal P}$.  The stronger quotient maps arcs of $\Gamma$ to arcs of $\Gamma/{\cal P}$ and non-arcs of $\Gamma$ to non-arcs of $\Gamma/{\cal P}$.  This allows more combinatorial and symmetry information to lift from $\Gamma/{\cal B}$ to $\Gamma$ and more to project from $\Gamma$ to $\Gamma/{\cal P}$.  This is what we will exploit in the latter part of this section.

While our definition of a double coset digraph is more or less the usual one, it might be better, for the purposes of this section, to think of first fixing the group $G$, then choosing $S$, and then letting $H\le G$ such that $S = HSH$.  The next lemma shows that such a subgroup always exists, even if it must be $H = 1$, in which case a double coset digraph is isomorphic to a Cayley digraph.  In this manner a double coset digraph of $G$ exists for every $S$, but of course not all of these digraphs will have the same number of vertices.

\begin{lemma}\label{K exists}
Let $G$ be a group, $S\subseteq G$, and $K = \{g\in G:gS = Sg = S\}$.  Then $K\le G$ and $K$ is the maximum subgroup of $G$ for which $S = KSK$.
\end{lemma}

\begin{proof}
The only part of showing $K\le G$ that is not completely straightforward is showing that if $gS = Sg = S$, then $g^{-1}S = Sg^{-1} = S$.  As $1\cdot S = S\cdot 1 = S$, $S = g^{-1}g\cdot S = g^{-1}S$ and $S = S\cdot gg^{-1} = Sg^{-1}$.  Hence $K\le G$.  To see $K$ is the maximum subgroup of $G$ for which $S = KSK$, if $L\le G$ such that $S = LSL$, then $gS = Sg = S$ for all $g\in L$, and hence $L\le K$.
\end{proof}

We will need some additional terms.

\begin{definition}\label{equiv relation}
Let $\Gamma$ be a digraph.  Define an equivalence relation $R$ on $V(\Gamma)$ by $u\ R\ v$ if and only if the out- and in-neighbors of $u$ and $v$ are the same.  Then $R$ is an equivalence relation on $V(\Gamma)$. We say $\Gamma$ is {\bf irreducible} if the equivalence classes of $R$ are singletons, and {\bf reducible} otherwise.
\end{definition}

The equivalence relation in Definition \ref{equiv relation} was introduced for graphs by Sabidussi \cite[Definition 3]{Sabidussi1959}, and independently rediscovered by Kotlov and Lov\'asz \cite{KotlovL1996}, who call $u$ and $v$ {\bf twins}, and Wilson \cite{Wilson2003}, who calls reducible graphs {\bf unworthy}.  It is easy to see that a vertex-transitive digraph $\Gamma$ is reducible if and only if it can be written as a wreath product $\Gamma_1\wr\bar{K}_n$ for some positive integer $n\ge 2$.  Sabidussi observed in \cite{Sabidussi1964} that $R$ is a $G$-congruence for transitive groups $G\le\Aut(\Gamma)$.

\begin{definition}
Let $G$ be a group.  Define $\Reducible(G)$ to be the set of all Cayley digraphs of $G$ that are reducible.  That is, $\Reducible(G) = \{\Cay(G,S):\Cay(G,S) \cong \Gamma\wr\bar{K}_t, \Gamma {\rm\ a\ digraph,\ } t\ge 2\}$.  Let $\Cos(G)$ be the set of all loopless double coset digraphs of $G$ that are not Cayley digraphs of $G$.   That is,
$$\Cos(G) = \{\Cos(G,H,S):1 < H\le G, H\cap S = \emptyset, {\rm\ and\ }S = HSH\}.$$
\end{definition}

The condition that a double coset digraph $\Cos(G,H,S)$ is loopless means that $H\cap S = \emptyset$.  We will implicitly use the easy to show fact that if $H\le K$, $KSK = S$ and $H\cap S = \emptyset$, then $K\cap S = \emptyset$.  This follows from the also easy to show fact that if $KSK = S$, then $K\cap S = \emptyset$ or $K\cap S = K$.  Note that the digraphs in $\Cos(G)$ need not all have the same order.  If we followed the convention that $H$ is core-free in $G$, the next result would not be true for abelian groups, for example, as the only core-free subgroup of an abelian group is trivial.

\begin{theorem}
Let $G$ be a group.  Define $\gamma\colon\Cos(G)\to \Reducible(G)$ by $\gamma(\Cos(G,H,S)) = \Cay(G,S)$.  Then $\gamma$ is onto, and $\gamma(\Cos(G,H_1,S_1)) = \gamma(\Cos(G,H_2,S_2))$ if and only if $S_1 = S_2$ and $\la H_1,H_2\ra S\la H_1,H_2\ra = S$.
\end{theorem}

\begin{proof}
Note that $\gamma$ indeed maps $\Cos(G)$ to $\Reducible(G)$ as if $\Cos(G,H,S)\in \Cos(G)$, then $H\cap S = \emptyset$ and $S = HSH$.  By Theorem \ref{cosetwr} $\Cay(G,S)\cong \Cos(G,H,S)\wr \bar{K}_h$ where $h = \vert H\vert$.  Thus $\Cay(G,S)$ is reducible, and the image of $\gamma$ is contained in $\Reducible(G)$.  To see $\gamma$ is onto, let $\Cay(G,S)\in \Reducible(G)$. Then $\Cay(G,S)$  $\cong \Gamma_1\wr\bar{K}_t$, where $\Gamma_1$ is a digraph and $t\ge 2$.  Choose $\Gamma_1$ and $t$ so that $\Aut(\Cay(G,S)) \cong \Aut(\Gamma_1)\wr\Aut(\bar{K}_t)$, in which case $\Aut(\Cay(G,S)) \cong \Aut(\Gamma_1)\wr S_t$.  Then $\Aut(\Cay(G,S))$ has the lexi-partition ${\cal B}$ with respect to $S_t$ as a block system with blocks of size $t$.  As $G_L\le\Aut(\Cay(G,S))$, the blocks of ${\cal B}$ are the left cosets of some subgroup $K$ of $G$.  As $\Cay(G,S) = \Gamma_1\wr\bar{K}_t$ and ${\cal B}$ is the lexi-partition, $\Cay(G,S)[gK]\cong\bar{K}_t$ for every $g\in G$ and so $S\cap K = \emptyset$.  By Lemma \ref{K determined lemma}, $S = K(S \setminus K)K = KSK$.  Then $\Cos(G,K,S)\in\Cos(G)$ and $\gamma(\Cos(G,K,S)) = \Cay(G,S)$.  Thus $\gamma$ is onto.

To see $\gamma(\Cos(G,H_1,S_1)) = \gamma(\Cos(G,H_2,S_2))$ if and only if $S_1 = S_2$ and $\la H_1,H_2\ra S\la H_1,H_2\ra = S$, let $\Cos(G,H_1,S_1),\Cos(G,H_2,S_2)\in\Cos(G)$.  Suppose $\gamma(\Cos(G,H_1,S_1)) = \gamma(\Cos(G,H_2,S_2))$.  Then $\Cay(G,S_1) = \Cay(G,S_2)$ and as $S_1$ and $S_2$ are both the neighbors of $1_G$, $S_1 = S_2$.  Set $S = S_1 = S_2$.  Then $H_1SH_1 = S$, $H_2SH_2 = S$, and it is then easy to see that $\la H_1,H_2\ra S \la H_1,H_2\ra = S$.  The converse is similarly difficult.
\end{proof}

As, by Lemma \ref{K exists}, for every group $G$ and $S\subset G$ there is a maximal (so unique) subgroup $K$ with $KSK = S$, we may restrict the domain of $\gamma$ to these unique subgroups $K$, and obtain a bijection.

\begin{corollary}\label{bijection}
Let $G$ be a group.  Let $\Cos_U(G)$ be the set of all loopless double coset digraphs of $G$ with connection set $S$ such that $K$ is chosen to be maximal with $KSK = S$.  That is, $\Cos_U(G) = \{\Cos(G,K,S):K\le G, S = KSK {\rm\ and\ }g_1Sg_2\not = S{\rm\ for\ some\ }  g_1,g_2\in G\setminus K\}$.  Define $\gamma_U\colon\Cos_U(G)\to \Reducible(G)$ by $\gamma_U(\Cos(G,H,S)) = \Cay(G,S)$.  Then $\gamma_U$ is a bijection.
\end{corollary}

\begin{theorem}\label{alt coset definition}
Let $G$ be a group, $H\le G$, and $S\subset G$ such that $S\cap H = \emptyset$.  Then $\Cos(G,H,S)$ is a well-defined double coset digraph of $G$ if and only if the set of equivalence classes of $R$ in $\Cay(G,S)$ is refined by $G/H$. Additionally, if $\Cos(G,H,S)$ is a well-defined double coset digraph of $G$, then $\Cos(G,H,S) = \Cay(G,S)/(G/H)$.
\end{theorem}

\begin{proof}
If the set of equivalence classes of $R$ in $\Cay(G,S)$ is refined by $G/H$, then $\Cay(G,S)$ can be written as a wreath product $\Gamma_1\wr\Gamma_2$, and $\Gamma_2$ is the empty digraph on $H$. Then there is a maximum supergraph $\Gamma_2'$ of $\Gamma_2$ that is an empty graph and $\Cay(G,S)\cong \Gamma_1'\wr\Gamma_2'$ for some $\Gamma_1'$.  As $\Gamma_1'\wr\Gamma_2'$ is isomorphic to a Cayley digraph, we may assume without loss of generality that $\Aut(\Gamma_1'\wr\Gamma_2')$ contains $G_L$.  As $\Gamma_2'$ is maximum, $\Gamma_1'$ is irreducible so $\Aut(\Gamma_1'\wr\Gamma_2')\cong\Aut(\Gamma_1')\wr\Aut(\Gamma_2')$ by Theorem \ref{autwr}.  As $G_L\le\Aut(\Gamma_1'\wr\Gamma_2')$, this gives that the lexi-partition of $\Aut(\Gamma_1'\wr\Gamma_2')$ with respect to $\Aut(\Gamma_2')$ is the set of left cosets of a supergroup $K$ of $H$, and as $\Gamma_2'$ is a supergraph of $\Gamma_2$, $H\le K$.  As $\Gamma_2'$ was chosen to be maximum, $K$ is the maximum subgroup of $G$ for which $\Gamma_2'$ can be chosen to have vertex set a subgroup of $G$.  Applying $\gamma_U^{-1}$ as defined in Corollary \ref{bijection}, we see $\gamma_U^{-1}(\Cay(G,S)) = \Cos(G,H,S)$ and $S = KSK$.  So $S = HSH$ and $\Cos(G,H,S)$ is a well-defined double coset digraph of $G$.

If $\Cos(G,H,S)$ is a well-defined double coset digraph, then by definition, $S = HSH$ and so $\Cay(G,S)$ is reducible by Corollary \ref{nonabelwr}.  Also, $\Cay(G,S) \cong (\Cay(G,S)/(G/H)) \wr \Gamma[H]$.  As $S\cap H = \emptyset$, $\Gamma[H]$ has no arcs.  This completes the if and only if statement of the result.  To finish, simply observe $\Cay(G,S)/(G/H)\cong\Cos(G,H,S)$.
\end{proof}

As Theorem \ref{alt coset definition} is an ``if and only if" it gives an alternative definition of a double coset digraph as follows:

\begin{definition}\label{coset alt definition}
Let $G$ be a group and $S\subset G$ such that $\Cay(G,S)$ is reducible with the equivalence classes of $R$ the left cosets of $K\le G$.  Let $H\le K$.  Define $\Cos(G,H,S)$ to be the digraph $\Cay(G,S)/(G/H)$.
\end{definition}

Let $G$ be a group, $1 < H < G$, and $S\subset G$.  Theorem \ref{alt coset definition} also gives an alternative way of computationally checking whether $S$ is a union of double cosets of $H$.  If $S$ is a union of double cosets of $H$, then Theorem \ref{alt coset definition} gives that $\Cay(G,S)$ is reducible.  This means that the equivalence classes of $R$ are not singleton sets.  Thus one only need check if, say, $1_G$ has the same set of in- and out-neighbors as some other vertex in $\Cay(G,S)$.

While this method of determining whether a subset $S\subset G$ defines a double coset digraph, once one has established that $\Cos(G,H,S)$ is well defined (even by checking $S = HSH$), the maximum $K$ for which $S = KSK$ is simply the equivalence class of $R$ which contains $1_G$.  From a group theoretic point of view, it is also easy to calculate $K$, as Lemma \ref{K exists} shows.

The next result generalizes \cite[Theorem 4]{Sabidussi1964}, which states that given a double coset digraph $\Gamma = \Cos(G,H,S)$ of a group $G$, there is a ``multiple" (this is Sabidussi's terminology) of $\Gamma$ that is a Cayley digraph of $G$.  The result gives the double coset digraphs $\Gamma$ of $G$ for which $\Cay(G,S)$ is a ``multiple" of $\Gamma$.

\begin{theorem}\label{wreath quotient coset}
Let $G$ be a group, $S\subset G$, and $K\le G$ maximal such that $S = KSK$.  Assume $S\cap K = \emptyset$.   Let $M\le K$, and $n_M = \vert M\vert$.  Then $$\Cos(G,M,S)\wr\bar{K}_{n_M}\cong (\Cay(G,S)/(G/M))\wr\bar{K}_{n_M}\cong \Cay(G,S).$$
\end{theorem}

\begin{proof}
First observe that as $M\le K$ and $S = KSK$, we have $S = MSM$.   So $\Cos(G,M,S)$ is well-defined. We apply Theorem \ref{cosetwr} with $H$ of that result the subgroup $\{1_G\}$, and $K$ of that result $M$ of this result. This gives
\begin{eqnarray*}
\Cos(G,\{1_G\},S) & = & \Cay(G,S)\cong \Cos(G,M,S\setminus M) \wr (\Cay(G,S)[M])\\
            & \cong & \Cay(G,S)/(G/M)\wr(\Cay(G,S)[M]).
\end{eqnarray*}
As $\Cay(G,S)$ is loopless and $S = MSM$, $1_G\not\in S$ and so $M\cap S = \emptyset$.  Then $\Cay(G,S)[M]\cong \bar{K}_{n_M}$, and the result follows.
\end{proof}

We now give the relationship between the automorphism groups of $\Cay(G,S)$ and $\Cos(G,H,S)$.  Determining the automorphism group of a vertex-transitive digraph is perhaps the most fundamental problem regarding vertex-transitive digraphs.  Corollary \ref{nonabelwr} shows that  $\Aut(\Cay(G,S))$ can be determined from $\Aut(\Cos(G,H,S))$, and the next result shows $\Aut(\Cos(G,H,S))$ can be obtained from $\Aut(\Cay(G,S))$.  So the problem of determining the automorphism groups of double coset digraphs of $G$ is equivalent to the problem of determining the automorphism groups of Cayley digraphs of $G$!  As every vertex-transitive digraph is isomorphic to a double coset digraph \cite[Theorem 2]{Sabidussi1964}, the problem of determining the automorphism group of vertex-transitive digraphs is equivalent to the problem of finding the automorphism groups of Cayley digraphs.  We need one more idea.

\begin{definition}
Let $G\le S_n$ be transitive with block systems ${\cal B}$ and ${\cal C}$.  We write ${\cal B}\preceq{\cal C}$ if every block of ${\cal C}$ is a union of blocks of ${\cal B}$, and say ${\cal B}$ {\bf refines} ${\cal C}$.
\end{definition}

\begin{corollary}\label{main cor}
Let $G$ be a group, $H\le G$, and $S\subset G$ such that $S = HSH$ and $S\cap H = \emptyset$. If $H\le K$ is chosen to be maximal such that $S = KSK$, ${\cal B}$ is the set of left cosets of $K$ in $G$, and $n = [K:H]$, then $$\Aut(\Cos(G,H,S)) \cong (\Aut(\Cay(G,S))/{\cal B})\wr S_n.$$
\end{corollary}

\begin{proof}
By Theorem \ref{wreath quotient coset} we have $$\Cos(G,K,S)\wr\bar{K}_{\vert K\vert}\cong (\Cay(G,S)/{\cal B})\wr\bar{K}_{\vert K\vert}\cong \Cay(G,S)$$
and
$$\Cos(G,H,S)\wr\bar{K}_{\vert H\vert}\cong \Cay(G,S).$$
As $\Cay(G,S)[K] = \bar{K}_{\vert K\vert}$ and choice of $K$, by Corollary \ref{nonabelwr}, we know that $$\Aut(\Cay(G,S)) \cong\Aut(\Cos(G,K,S))\wr S_{\vert K\vert}.$$
Hence $\Aut(\Cos(G,K,S)) = \Aut(\Cay(G,S))/{\cal B}$.  Let ${\cal C}$ be the set of left cosets of $H$ in $G$, and $A\le\Aut(\Cay(G,S))$ be maximal that has ${\cal C}$ as a block system.  Then $A = \Aut(\Cos(G,K,S))\wr S_n\wr S_{\vert H\vert}$.  As $\Cos(G,H,S)\wr\bar{K}_{\vert H\vert}\cong \Cay(G,S)$ and $\Aut(\Cos(G,K,S)) = \Aut(\Cay(G,S))/{\cal B}$, we have
\[\Aut(\Cos(G,H,S)) = A/{\cal C} = \Aut(\Cos(G,K,S))\wr S_n = (\Aut(\Cay(G,S))/{\cal B})\wr S_n. \qedhere \]
\end{proof}

\begin{corollary}\label{main cor two}
Let $G$ be a group, $H\le G$, and $S\subset G$ such that $S = HSH$ and $S\cap H = \emptyset$. If $H\le K$ is chosen to be maximal such that $S = KSK$, ${\cal C}$ is the set of left cosets of $H$ in $G$, and $k = \vert K\vert$, then $$\Aut(\Cay(G,H,S)) \cong (\Aut(\Cos(G,H,S))/({\cal B}/{\cal C}))\wr S_k.$$
\end{corollary}

\begin{proof}
Let ${\cal B}$ be the set of left cosets of $K$ in $G$.  As $K$ is chosen to be as large as possible, $\Cay(G,S)/{\cal B}$ is irreducible.  Hence $\Aut(\Cay(G,S)) = (\Aut(\Cay(G,S)/{\cal B})\wr S_k$.  As ${\cal C}$ is a refinement of ${\cal B}$ as $H\le K$, and the vertex set of $\Cos(G,H,S)$ is the set of left cosets of $H$ in $G$, $$\Cos(G,H,S)/({\cal B}/{\cal C})\cong\Cos(G,H,S)/{\cal B}.$$
\end{proof}

Our next goal is to show that the isomorphism problem can be solved for all double coset digraphs of $G$ if and only if it can be solved for the corresponding Cayley digraphs of $G$.  We will need a preliminary lemma.

\begin{lemma}\label{wreath iso}
Let $G$ be a group, $H_i\le K_i\le G$ such that $\vert H_1\vert = \vert H_2\vert$ and $\vert K_1\vert = \vert K_2\vert$, and $S_i\subset G$ such that $K_i\cap S_i = \emptyset$ and $K_iS_iK_i = S_i$, $i = 1,2$.  Suppose $\delta\colon G/K_1\to G/K_2$ is an isomorphism between $\Cos(G,K_1,S_1)$ and $\Cos(G,K_2,S_2)$.  Define a map $\tilde{\delta}$ from $G/H_1$ to $G/H_2$ in the following way:  first, $\tilde{\delta}(gK_1) = \delta(gK_1)$ for every $g\in G$.  That is, $\tilde{\delta}$ maps $G/K_1$ to $G/K_2$ in the same fashion as $\delta$. To extend $\delta$ to $\tilde{\delta}$, map elements of $gK_1/H_1$, $g\in G$, to the set of left cosets of $H_2$ in the left coset $\delta(gK_1)$ of $K_2$ in $G$ bijectively in any fashion.  Then $\tilde{\delta}$ is an isomorphism from $\Cos(G,H_1,S_1)$ to $\Cos(G,H_2,S_2)$.
\end{lemma}

\begin{proof}
Note that as $K_iS_iK_i = S_i$, $H_iS_iH_i = S_i$, and, as $S_i\cap K_i = \emptyset$, $S_i\cap H_i = \emptyset$.  Hence $\Cos(G,H_i,S_i)$ is a well-defined coset digraph, $i = 1,2$.  Then $[K_1:H_1] = [K_2:H_2] = n$ as $\vert H_1\vert = \vert H_2\vert$ and $\vert K_1\vert = \vert K_2\vert$.  Also, $\tilde{\delta}$ is a bijection as it maps $G/K_1$ to $G/K_2$ and maps the left cosets of $H_1$ contained in $gK_1$ to the left cosets of $H_2$ contained in  $\delta(gK_1)$ for every $g\in G$.  By Theorem \ref{cosetwr}, we see $\Cos(G,H_i,S_i)\cong \Cos(G,K_i,S_i)\wr \bar{K}_n$, $i = 1,2$.  Let $(xH_1,yH_1)\in A(\Cos(G,H_1,S_1))$.  Then $(xK_1,yK_1)\in A(\Cos(G,K_1,S_1))$, and, as $K_1\cap S_1 = \emptyset$, we see $xK_1\not = yK_1$.  Then $(\delta(xK_1),\delta(yK_1)) = (aK_2,bK_2)\in A(\Cos(G,K_2,S_2))$ for some $a,b\in G$.  As $xK_1\not = yK_1$, $aK_2\not = bK_2$.  Hence $(ak_2H_2,bk_2'H_2)\in A(\Cos(G,H_2,S_2))$ for every $k_2,k_2'\in K_2$ so $\tilde{\delta}(xH_1,yH_1)\in A(\Cos(G,H_2,S_2))$, and $\tilde{\delta}$ is an isomorphism from $\Cos(G,H_1,S_1)$ to $\Cos(G,H_2,S_2)$.
\end{proof}

\begin{theorem}
Let $G$ be a group, $H_i\le G$ such that $\vert H_1\vert = \vert H_2\vert$, and $S_i\subset G$ such that $H_i\cap S_i = \emptyset$ and $H_iS_iH_i=S_i$, $i = 1,2$.  Then $\Cos(G,H_1,S_1)\cong \Cos(G,H_2,S_2)$  if and only if $\Cay(G,S_1)\cong \Cay(G,S_2)$.
\end{theorem}

\begin{proof}
If $\Cos(G,H_1,S_1)\cong\Cos(G,H_2,S_2)$ then clearly their vertex-sets have the same cardinality and so $\vert H_1\vert = \vert H_2\vert$ (i.e. the hypothesis that $\vert H_1\vert = \vert H_2\vert$ is not needed for this part of the proof).  The result follows from Lemma \ref{wreath iso} (with $K_i$ and $H_i$ of that lemma $H_i$ and $\{1\}$, respectively).

Now suppose $\delta\colon G\to G$ is an isomorphism between $\Cay(G,S_1)$ and $\Cay(G,S_2)$. Let $K_i\le G$ be maximal such that $S_i = K_iS_iK_i$, $i = 1,2$. Then $H_i\le K_i$, $i = 1,2$.  By Theorem \ref{nonabelwr}, we have by choice of $K_i$ that $\Aut(\Cay(G,S_i))\cong\Aut(\Cos(G,K_i,S_i))\wr\bar{K}_{\vert K_i\vert}$, $i = 1,2$.  By Theorem \ref{autwr} we see that $\Cos(G,K_i,S_i)$ cannot be written as a nontrivial wreath product with the complement of a complete graph, so $\Cos(G,K_i,S_i)$ is irreducible, $i = 1,2$.  This in turn implies $\vert K_1\vert = \vert K_2\vert$.  As $G/K_i$ is the lexi-partition of $\Cay(G,S_i)$ with respect to $\bar{K}_{\vert K_i\vert}$, $G/K_i$ is a block system of $\Aut(\Cay(G,S_i))$, $i = 1,2$, and $\delta(G/K_1)$ is a block system of $\Aut(\Cay(G,S_2))$.  Then $\delta(G/K_1)\preceq G/K_2$ or $G/K_2\preceq\delta(G/K_1)$ by \cite[Lemma 5]{DobsonM2015a}, and so $\delta(G/K_1) = G/K_2$.  The result now follows by Lemma \ref{wreath iso}.
\end{proof}

The preceding result may appear to reduce the isomorphism problem for vertex-transitive digraphs to the isomorphism problem for Cayley digraphs, as every vertex-transitive graph can be written as a double coset digraph.  This, however, is not the case, as it is quite possible that a Cayley digraph is isomorphic to a Cayley digraph of more than one group, see for example \cite{Morris1999}.  The preceding result will not give the isomorphisms between two different representations of a single digraph as Cayley digraphs on different groups.

\begin{corollary}
Let $G_1$ and $G_2$ be groups with $\vert G_1\vert = \vert G_2\vert$, $H_i\le G_i$ with $\vert H_1\vert = \vert H_2\vert$, and $S_i\subset G_i$ such that $H_i\cap S_i = \emptyset$ and $H_iS_iH_i=S_i$, $i = 1,2$. Let $\{L_1,\ldots,L_t\}$ be the set of all regular subgroups of $\Aut(\Cay(G_1,S_1))$. Then $\Cos(G_1,H_1,S_1)\cong\Cos(G_2,H_2,S_2)$ if and only if there is some $1\le j\le t$ such that $L_j \cong G_2$ and $\Cay(G_2,S_2)\cong \Cay(L_j,T_j)$ for some $T_j\subseteq L_j$.
\end{corollary}

\begin{proof}
As $\vert G_1\vert = \vert G_2\vert$ and $n = \vert H_1\vert = \vert H_2\vert$ we have that $\vert V(\Cos(G_1,H_1,S_1))\vert = \vert V(\Cos(G_2,H_2,S_2))\vert$. As $H_iS_iH_i = S_i$, by Corollary \ref{nonabelwr}, we have that $\Cos(G_i,H_i,S_i)\wr\bar{K}_n\cong\Cay(G_i,S_i)$, $i = 1,2$.

Suppose $\Cos(G_1,H_1,S_1)\cong\Cos(G_2,H_2,S_2)$.   Then $\Cos(G_1,H_1,S_1)\wr\bar{K}_n\cong \Cos(G_2,H_2,S_2)\wr\bar{K}_n$.  As $\Cos(G_i,H_i,S_i)\wr\bar{K}_n\cong\Cay(G_i,S_i)$,  $\Aut(\Cay(G_1,S_1))$ contains a regular subgroup isomorphic $G_2$, so $G_2\cong L_j$ for some $L_j$.  Then there exists $T_j\subseteq L_j$ with $\Cay(G_2,S_2)\cong \Cay(L_j,T_j)$ for some $1\le j\le t$.

Conversely, suppose there is some $1\le j\le t$ such that $L_j\cong G_1$ and $\Cay(G_2,S_2)\cong\Cay(L_j,T_j)$.  Then $\Cay(G_2,S_2)\cong\Cay(G_1,S_1)$.  As $\Cos(G_i,H_i,S_i)\wr\bar{K}_n\cong\Cay(G_i,S_i)$, $i = 1,2$, we have $\Cos(G_1,H_1,S_1)\cong\Cos(G_2,H_2,S_2)$.
\end{proof}

In particular, this shows the importance of the problem of determining when a Cayley digraph is isomorphic to a Cayley digraph of more than one group.

\section{Generalized Wreath Products}

Generalized wreath products are a fairly new class of digraphs that were introduced to describe automorphism groups of circulant digraphs (Cayley digraphs of cyclic groups).  They form one of three broad families of digraphs (the others being deleted wreath product types whose automorphism group has a factor a symmetric group, and normal Cayley digraphs of $\Z_n$), and it was stated in \cite[Theorem 2.3]{Li2005} that all circulant digraphs fall into at least one of these families.  This result is a translation of results proven using results on Schur rings \cite{EvdokimovP2002,LeungM1996,LeungM1998} to the language of vertex-transitive digraphs.  We will not dwell on a fourth family, namely those with primitive automorphism groups, as for circulant digraphs these digraphs are only the complete graph and its complement. Normal Cayley digraphs were introduced in a very nice paper \cite{Xu1998} by M.Y. Xu in 1998 and are those Cayley digraphs of $G$ for which $G_L\tl\Aut(\Cay(G,S))$.  Deleted wreath type digraphs were first defined in \cite{BhoumikDM2014}, and are those digraphs whose automorphism group is the same as a deleted wreath product of two smaller digraphs (in general this family should probably be those digraphs whose automorphism group has a factor which is the automorphism group of a digraph of smaller order that is quasiprimitive).

\begin{definition}
 Let $\Gamma_1$ and $\Gamma_2$ be digraphs. The \textbf{deleted wreath product of $\Gamma_1$ and $\Gamma_2$}, denoted $\Gamma_1 \wr_d \Gamma_2$, is the digraph with vertex set $V(\Gamma_1) \times V(\Gamma_2)$ and arc set $$\{((x_1, y_1),(x_2, y_2)) : (x_1, x_2) \in A(\Gamma_1) \text{ and } y_1 \not= y_2 \text{ or } x_1 = x_2 \text{ and } (y_1, y_2) \in A(\Gamma_2)\}.$$
\end{definition}

The determination of automorphism groups of circulant digraphs is not quite complete, although a polynomial time algorithm to determine a generating set of the automorphism group is known \cite{Ponomarenko2005} (which for many purposes does provide a complete solution).  While we have a classification of circulant digraphs into the three families mentioned above, the automorphism groups of deleted wreath products are not known at this time, although some partial results are given in \cite{DobsonMS2019}.  These results do though give a general template on how to approach the problem of finding automorphism groups of other classes of vertex-transitive digraphs.  Namely, prove a classification type result to show that all Cayley digraphs under consideration fall within a certain set of families of digraphs.  Then one should determine the automorphism groups of the digraphs in each of the families, and a complete determination of the automorphism groups will be obtained.

Generalized wreath products have thus far only been defined for Cayley digraphs of abelian groups.  The reason that they were not defined for Cayley digraphs of nonabelian groups or vertex-transitive digraphs that are not Cayley digraphs is that the recognition problem for when such digraphs are wreath products had not been solved.  The idea behind a generalized wreath product $\Gamma$ is that we do not have any control over which arcs have both endpoints inside a block of a block system ${\cal C}$ of a transitive subgroup of the automorphism group, but the other arcs  form a digraph that is a wreath product, and the lexi-partition of that wreath product refines ${\cal C}$.  Thus we want to be able to decompose the arc set of the digraph into two sets in such a way that one set of arcs defines a disconnected digraph  (which is a wreath product) and the other a wreath product in such a way that the automorphism group of $\Gamma$ contains the intersection of the automorphism groups of  the two wreath products.  As up to now we could not determine, by inspection of the connection set, whether the remaining arcs formed a wreath product, we could not extend the definition to other vertex-transitive digraphs.  Correcting this defect was the original motivation for this paper.  We now define generalized wreath products for all double coset digraphs, and recall that by \cite[Theorem 2]{Sabidussi1964} that every vertex-transitive digraph is isomorphic to a double coset digraph.

\begin{definition}
Let $G$ be a group with subgroups $1 \le H < K \le L < G$ and $S\subset G$ a union of double cosets of $H$ in $G$ such that $S\setminus L$ is a union of double cosets of $K$.  The double coset digraph $\Cos(G,H,S)$ is called a {\bf $(K,L)$-generalized wreath product}.
\end{definition}

Notice that if $K = L$, then by Theorem \ref{cosetwr} we have that $\Cos(G,H,S)$ is a wreath product.  So generalized wreath products are a generalization of the wreath product.   It is straightforward to show that if $\Cos(G,H,S)$ is a $(K,L)$-generalized wreath product, then $\Aut(\Cos(G,H,S))$ does contain the intersection of $\Aut(\Cos(G,H,L\cap S))$ and $\Aut(\Cos(G,H,S\setminus L))$, and these digraphs have the properties we were aiming for.

\begin{lemma}
Let $G$ be a group, $S\subset G$, and $1\le H < K\le L < G$.  If $\Cos(G,H,S)$ is isomorphic to a $(K,L)$-generalized wreath product then
$$\Aut(\Gamma)\ge \Aut(\Cos(G,H,S\cap L))\cap\Aut(\Cos(G,H,S\setminus L))\ge (S_r\wr\Gamma[L])\cap (\Gamma/{\cal B}\wr S_t),$$
where $r = [G:L]$, $t = \vert K\vert$, and ${\cal B}$ is the set of left cosets of $H$ in $G$.
\end{lemma}

As a good general rule is that symmetry in digraphs is rare, one expects that the automorphism group of a $(K,L)$-generalized wreath product would be $\Aut(\Cos(G,H,S\cap L))\cap\Aut(\Cos(G,H,S\setminus L))$.  It seems likely that there will be many ways in which the automorphism group will be larger than expected.  The following problem is then natural, and its solution is a crucial step in determining automorphism groups of vertex-transitive digraphs.

\begin{problem}
Determine necessary and sufficient conditions for the automorphism group of a $(K,L)$-generalized wreath product $\Cos(G,H,S)$ to have automorphism group $\Aut(\Cos(G,H,S\cap L))\cap\Aut(\Cos(G,H,S\setminus L))$ as expected.  Additionally, for each class of generalized wreath products that do not have automorphism group $\Aut(\Cos(G,H,S\cap L))\cap\Aut(\Cos(G,H,S\setminus L))$, determine the full automorphism group for each such class.
\end{problem}

It has been shown in \cite[Theorem 35]{DobsonM2005} that a $(K,L)$-generalized wreath product circulant digraph of square-free order $n$ has automorphism group $\Aut(\Cay(\Z_n,S\cap L))\cap\Aut(\Cay(\Z_n,S\setminus L))$, and extended to all circulant digraphs in \cite[Lemma 3.4]{AraujoBDKM2018}.  Also, Theorem \ref{autwr} solves this problem in the special case when $\Gamma$ is a wreath product, which is guaranteed when $K = L$.


\providecommand{\bysame}{\leavevmode\hbox to3em{\hrulefill}\thinspace}
\providecommand{\MR}{\relax\ifhmode\unskip\space\fi MR }
\providecommand{\MRhref}[2]{%
  \href{http://www.ams.org/mathscinet-getitem?mr=#1}{#2}
}
\providecommand{\href}[2]{#2}

\end{document}